\documentclass[journal]{IEEEtran}
\IEEEoverridecommandlockouts
  \usepackage[noadjust]{cite}
  \usepackage{amsmath}
  \usepackage{amsthm}
  \usepackage{amssymb}
  \usepackage{graphicx}
  \usepackage{booktabs}
  \usepackage{threeparttable}
  \usepackage{mathtools}
  \usepackage{multirow}
  \usepackage{color,soul}
  \usepackage{url}
  \usepackage{epstopdf}
  \usepackage{algorithm}
  \usepackage{algorithmicx}
  \usepackage{algpseudocode}
  \usepackage{subfigure}

  \renewcommand{\vec}[1]{\boldsymbol{#1}}
  \newcommand{\tc}{t_{\rm c}}
  \newcommand{\te}{t_{\rm e}}
  \newcommand{\T}{\Delta\mathrm{T}}
  
  \newtheorem{remark}{Remark}

  \ifCLASSINFOpdf
  \else
  \fi

\begin{document}

  \title{Collaborative Autonomous Optimization of Interconnected Multi-Energy Systems with Two-Stage Transactive Control Framework}

\author{Yizhi Cheng, Peichao Zhang, Xuezhi Liu \\
\thanks{This research was funded by the National Key R\&D Program of China (Basic Research Class 2018YFB0905000).

Yizhi Cheng, Peichao Zhang and Xuezhi Liu are with the Department of Electrical Engineering, Shanghai Jiao Tong University, Shanghai, 200240 China (emails: tabdel@sjtu.edu.cn; pczhang@sjtu.edu.cn; liuxz@sjtu.edu.cn).
}}

  \maketitle

  \begin{abstract}

  Motivated by the benefits of multi-energy integration, this paper establishes a bi-level two-stage framework based on transactive control, in order to achieve optimal energy provision among interconnected multi-energy systems (MESs). At the lower level, each MES autonomously determines the optimal set points of each controllable assets by solving a cost minimization problem, in which rolling horizon optimization is adopted to deal with load and renewable energies’ stochastic features. A technique is further implemented for optimization model convexification by relaxing storages’ complementarity constraints, and its mathematical proof verifies the exactness of the relaxation. At the upper level, a coordinator is established to minimize total costs of collaborative interconnected MESs while preventing transformer overloading. This collaborative problem is further decomposed and solved iteratively in a two-stage procedure based on market-clearing mechanism. A distinctive feature of the method is that it is compatible with operational time requirement, while retaining scalability, information privacy and operation authority of each MES. Effectiveness of the proposed framework is verified by simulation cases that conduct detailed analysis of the autonomous-collaborative optimization mechanism.

  \end{abstract}


  \IEEEpeerreviewmaketitle

  \section{Introduction}

  \IEEEPARstart {C}{ontinuous} environment deterioration and energy depletion have necessitated the comprehensive utilization of various forms of energy. Accordingly, multi-energy system (MES) has gained significant attention for its ability to improve comprehensive energy efficiency, as well as to benefit system economy and environment \cite{7842813}. Recent years have also witnessed a research re-orientation from energy optimization of a standalone MES to the collaborative optimization among multiple interconnected MESs (IMESs) \cite{CHEN2018403}. For one thing, IMESs are able to shift supply and demand across energy vectors and networks, as well as to handle the uncertain and volatile generation outputs of renewable energy sources (RES) \cite{8534390, 6872821}. In addition, owing to the energy complementarity, the collaboration of MESs have the untapped potentials to jointly minimize the operational cost \cite{LIU-ACCESS2018}, enhance the overall operation efficiency and to increase system flexibility \cite{CHEN2018403}.

  Collaboration methods of interconnected subsystems generally fall into centralized and distributed mechanisms. The centralized models are established in \cite{CHEN2018403, EN2448,BEIGVAND20171090, LIU-ACCESS2018} and subsequently solved with traditional mathematical algorithms \cite{EN2448, CHEN2018403} or modern intelligent optimization techniques \cite{BEIGVAND20171090, LIU-ACCESS2018}. Although such optimizations guarantee a global optimal utilization of resources, they fail to meet privacy protection, scalability and openness requirements. Besides, MESs may have different owners and schedule resources based on their own economic rules and policies \cite{7812781,CHEN2018403}. This indicates that a direct decision authority or compulsory dispatch orders can be unpractical.

  Concerning this, the distributed mechanism is more favorable. While ensuring satisfactory solutions, it is capable of protecting crucial information of individual entities \cite{8534390}. Two commonly used methods of distributed mechanism are cooperative game theory \cite{8004502,8017431}, and alternating direction method of multipliers (ADMM) algorithm \cite{8283573, en11102555}. Although these methods address the issue of scalability, the coordination signals used in these methods, e.g., the Lagrange multipliers, do not provide explicit market information in collaboration \cite{8534390}. Compared with these methods, transactive energy (TE) developed in recent years uses price as a key operational parameter in collaboration \cite{PNNL22946}. It has been used in several applications, such as the Grid-SMART demonstration project \cite{gridSMART} and the PowerMatching City project \cite{Kok2005}, as well as in some researches on coordinating networked microgrids \cite{8534390,8501585,8486723}.

  Although the TE framework has been successfully applied to optimize the operation of IMESs, two prominent issues still need to be addressed.

  First, the existing TE methods usually require a comparable number of iterations to converge \cite{8534390}, which makes it unpractical for hourly-scheduling or real-time optimization in respect of communication latency, throughput and distortion. 
  Considering the uncertainty and unpredictability of RES output and load demand profile, a collaboration framework for hourly-scheduling is highly demanded, so as to better integrate high penetration of RES and improve overall energy efficiency. Although an adaptive scheme has been proposed in \cite{FATHI-SE2013} for power dispatching among networked microgrids in this aspect, it is not applicable to model with intertemporal constraints.
   
  Electric or thermal storages are becoming increasingly important in the MES. The second issue is concerned with constraints that prevent simultaneous charging and discharging of energy storages. To be specific, existence of these complementarity constraints leads to mix-integer programming that results in long solution time and excessive iterations due to additional binary variables \cite{7091047}. In some cases, these non-convex optimizations are even computationally intractable to solve \cite{7091047, 8592008}. Besides, sometimes a globally optimal solution is not guaranteed even with efficient algorithms. Although researchers of \cite{BECK2016331, 6837496, 8486723,CORTES201953} have illustrated that these complementarity constraints are redundant under normal operation mode, an exact definition of such normal cases is lacked in these researches. Three groups of sufficient conditions are proved in \cite{8592008} for exact relaxation of storage-concerned economic dispatch, however, those conditions are not generally suitable for this paper. Some preliminary work regarding this issue has been made in our previous paper \cite{8245724}, while this paper follows its footstep and provides proof with depth both mathematically and through simulation.

  To tackle these two problems, a two-stage bi-level transactive control (TC) framework is presented in this paper to realize a collaborative optimization of autonomous MESs. Its main contributions are:
  \begin{itemize}
    \item The bi-level collaborative problem of coordinating IMESs installed with storages is modeled with rolling horizon optimization to deal with stochastic features of RES and loads.
    \item An equivalent-energy-change transform is proposed for relaxing storage complementarity constraints, and one sufficient condition is provided and proved to guarantee the exact relaxation.
    \item A two-stage optimization procedure is designed where a dual-variable simplification technique is applied to solve the hourly-scheduling problem within limited iterations, thus meeting the operational time requirement.
  \end{itemize}

  The paper is structured as follows. The general framework is firstly described in Section II. Section III then develops the detailed MES-level autonomous optimization. The two-stage bi-level transactive energy optimization procedure is then proposed in Section IV and tested in Section V. Finally, Section VI concludes the paper.

  \section{Bi-level Optimization Framework}

  \subsection{Energy Management Entities and System Architecture}

  This paper considers two energy management entities:

 \subsubsection{MES Operator}

 Energy systems have been tightly coupled nowadays, leading to the blooming of integrated energy service companies that manage several types of energy concurrently. Therefore, this paper assumes that an operator is responsible for optimal operation within the multi-energy system. However, as the interconnected network among IMESs cannot be monitored by individual MESs, each MESs' optimized solutions might not meet the requirement of secure IMESs operation.

  \subsubsection{System Coordinator}

  To address the above issue, this paper assumes that an upper-level entity, referred to as system coordinator, is responsible for coordinating the IMESs and managing the operation of interconnected network. Meanwhile, the system coordinator also serves as an interface between IMESs and the utility grid, by strategically responding to dispatching signals.

  Under these assumptions, a collaborative autonomous optimization framework is proposed based on the TC mechanism for energy management of IMESs, as shown in Fig.\ref{fig_framework}. In the proposed framework, each MES optimizes to minimize its operational cost autonomously, and incentive/responsive signals are exchanged between MESs and the system coordinator to achieve a collaborative optimization. By this way, issues can be addressed effectively that are associated with individual MESs' information privacy and operation authority among different management entities.

  \begin{figure}
    \centering
    \includegraphics[width=0.48\textwidth]{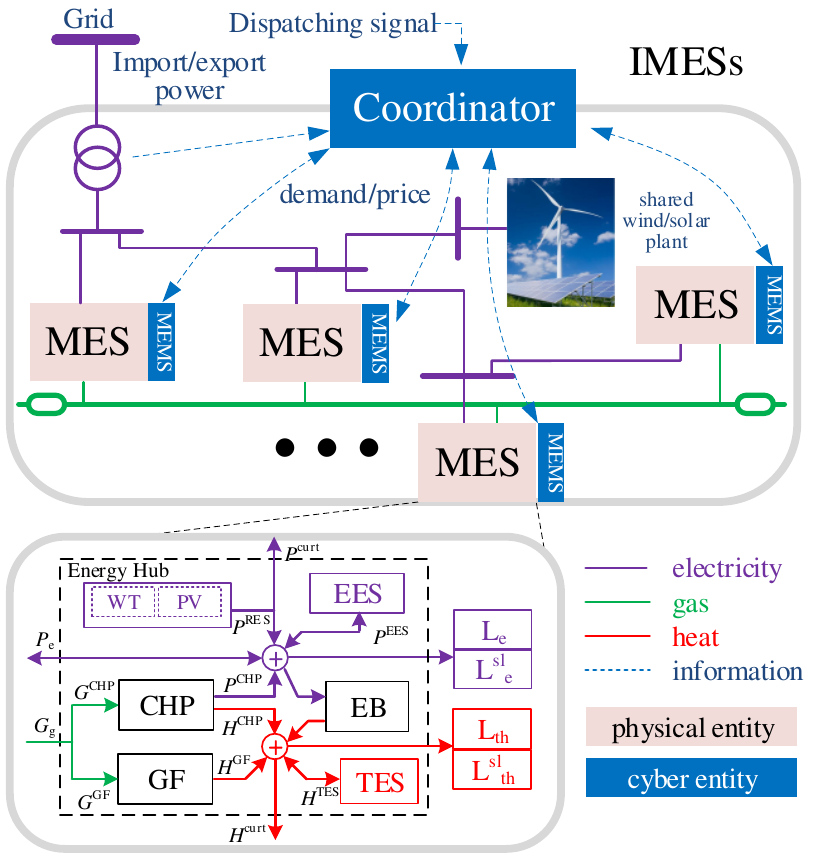}
    \caption{Energy management architecture for interconnected MESs.}
    \label{fig_framework}
  \end{figure}

  \subsection{General Assumptions}
  The following assumptions are made in this paper:
  \begin{itemize}
  \item The gas price remains unchanged throughout a day \cite{CHEN2018403}, since the gas price usually varies much slower than the electricity price.
  \item The real-time electricity prices are announced one day ahead or can be forecasted accurately.
  \item The forecast errors of RES output and load demand follow the normal distributions with parameters known beforehand.
  \end{itemize}

  \section{Autonomous Convex Optimization of MES}

  \subsection{MES Model}
  The lower part of Fig.\ref{fig_framework} exemplifies the structure of an MES consisting of a gas turbine combined heat and power (CHP) unit, a natural gas furnace (GF), an electric boiler (EB), an electric energy storage (EES) and a thermal energy storage (TES). Its model can be formulated as:
  \begin{equation}
    \begin{split}
    \label{eqn_MESmodel}
    \left(
      \begin{matrix}
      P_t + P_t^{\rm res}\\ 0
      \end{matrix}
    \right)
    +
    \left(
      \begin{matrix}
        \eta \rm _{ge}^{CHP} & 0 \\
        \eta {\rm _{gth}^{CHP}} & \eta \rm{_{gth}^{GF}}
      \end{matrix}
    \right)
    \left(
      \begin{matrix}
         G_t^{\rm CHP} \\ G_t^{\rm GF}
      \end{matrix}
    \right)
    +
    \left(
      \begin{matrix}
        -1 \\ \eta_{\rm eth}^{\rm EB}
      \end{matrix}
    \right)P_t^{\rm EB}
    + \\
    \left(
      \begin{matrix}
        P_{\mathrm{dch},t}^{\rm EES} - P_{\mathrm{ch},t}^{\rm EES} \\
        H_{\mathrm{dch},t}^{\rm TES} - H_{\mathrm{ch},t}^{\rm TES}
      \end{matrix}
    \right)
    -
    \left(
      \begin{matrix}
        P_t^{\rm curt} \\ H_t^{\rm curt}
      \end{matrix}
    \right)
    =
    \left(
      \begin{matrix}
        L_{\mathrm{e},t}^{\rm sl} \\ L_{\mathrm{th},t}^{\rm sl}
      \end{matrix}
    \right)
    +
    \left(
      \begin{matrix}
        L_{\mathrm{e},t} \\ L_{\mathrm{th},t}
      \end{matrix}
    \right),
    \end{split}
  \end{equation}
  where $P_t$ denotes the electricity power MES imports from the main gird at time ${t}$, and $P_t<0$ implies that the MES sells surplus electricity to the gird;
  $P_t^{\rm res}$ represents the onsite RES power generated under the maximum power point tracking mode;
  $G_t^{\rm CHP}$ and $G_t^{\rm GF}$ denote natural gas consumed by CHP and furnace, respectively;
  $\eta \rm _{ge}^{CHP}$ and $\eta \rm _{gth}^{CHP}$ are the gas-electric and gas-thermal efficiencies of CHP, and $P_t^{\rm EB}$ denote electricity consumed by boiler;
  $\eta \rm_{gth}^{GF}, \eta \rm_{eth}^{EB}$ are efficiencies of the furnace and boiler, respectively;
  $P_{\mathrm{ch},t}^{\rm EES}$ and $P_{\mathrm{dch},t}^{\rm EES}$ represent the charging and discharging power of EES, respectively, while $H_{\mathrm{ch},t}^{\rm TES}$ and $H_{\mathrm{ch},t}^{\rm TES}$ are the charging and discharging amount of TES, respectively;
  $P_t^{\rm curt}$ and $ H_t^{\rm curt}$ represent the curtailed RES and heat;
  $ L_{\mathrm{e},t}^{\rm sl}$ and $ L_{\mathrm{th},t}^{\rm sl}$ denote the shiftable electric load and thermal load, while $L_{\mathrm{e},t}$ and $L_{\mathrm{th},t}$ denote the fixed electric load and thermal load, respectively.

  \subsection{Formulation of Optimization Problem}
  During every scheduling period, after updating forecasts of local RES output and load demand, each MES seeks to minimize its expected cost across the remaining periods in an autonomous manner.
   
  The operational cost during period $t$ can be split into two parts: the electricity purchasing cost and the gas purchasing cost,
  \begin{equation}
    F_t = \mu _{\mathrm{e},t} P_t +
    \mu _{\mathrm{g},t}
    \left(G_{\mathrm{g},t}^{\rm CHP}
    + G_{\mathrm{g},t}^{\rm GF}\right),
  \end{equation}
  where $\mu _{\mathrm{e},t}$ and $\mu _{\mathrm{g},t}$ denote electricity price under the real-time pricing (RTP) scheme and natural gas price at period $t$, respectively.

  In addition to power balance constraints \eqref{eqn_MESmodel}, other constraints include:
  \begin{enumerate}
  \item The capacity constraints of electricity:
    \begin{equation}
      \label{eqn_electricity}
      -\overline P^{\rm out} \leq P_t \leq \overline P^{\rm in}, \forall t,
    \end{equation}
  where $\overline P^{\rm out} , \overline P^{\rm in} > 0$ denote the maximum exchange powers through the connecting line.

  \item The electric capacity limits of CHP, furnace and boiler:
  \begin{eqnarray}
    &\underline P^{\rm CHP} \leq P_t^{\rm CHP} = \eta_{\rm ge}^{\rm CHP} G_t^{\rm CHP} \leq \overline P^{\rm CHP}, \forall t \\
    &\underline P^{\rm GF} \leq P_t^{\rm GF} = \eta_{\rm gth}^{\rm GF} G_t^{\rm GF} \leq \overline P^{\rm GF}, \forall t\\
    &\underline P^{\rm EB} \leq P_t^{\rm EB} \leq \overline P^{\rm EB}, \forall t,
    \end{eqnarray}
    where $\overline P^{\rm CHP}$, $\overline P^{\rm GF}$ and $\overline P^{\rm EB}$ denote the installed capacity of the CHP, furnace and boiler respectively; $\underline P^{\rm CHP}$, $\underline P^{\rm GF}$ and $\underline P^{\rm EB}$ denote lower limits of corresponding electric power.
  \item The ramping constraints of CHP and boiler:
  \begin{eqnarray}
    |P_{t+1}^{\rm CHP} - P_{t}^{\rm CHP}| \leq \Delta P^{\rm CHP}, \forall t\\
    |P_{t+1}^{\rm EB} - P_{t}^{\rm EB}| \leq \Delta P^{\rm EB}, \forall t,\label{eqn_GF}
  \end{eqnarray}
  where $\Delta P^{\rm CHP}$ and $\Delta P^{\rm EB}$ denote the hourly ramping rate of CHP and boiler.
  \item The maximum charging/discharging power constraints:
  \begin{eqnarray}
    \label{eqn_EES_power_limits}
    \left\{
      \begin{aligned}
        0\leq P_{\mathrm{ch},t}^{\rm EES} \leq \overline P_{\rm ch}^{\rm EES}, \forall t \\
    0 \leq P_{\mathrm{dch},t}^{\rm EES} \leq \overline P_{\rm dch}^{\rm EES}, \forall t
      \end{aligned}
      \right.\\
    \label{eqn_TES_power_limits}
    \left\{
      \begin{aligned}
        0 \leq H_{\mathrm{ch},t}^{\rm TES} \leq \overline H_{\rm ch}^{\rm TES}, \forall t \\
        0 \leq H_{\mathrm{dch},t}^{\rm TES} \leq \overline H_{\rm dch}^{\rm TES}, \forall t,
      \end{aligned}
    \right.
  \end{eqnarray}
  where $\overline P_{\rm ch}^{\rm EES}$,$\overline P_{\rm dch}^{\rm EES}$ denote the maximum charging and discharging power of the EES, and $\overline H_{\rm ch}^{\rm TES}$, $\overline H_{\rm dch}^{\rm TES}$ denote the maximum charging and discharging quantity of the TES, respectively.

  \item The mutual exclusiveness of charging and discharging mode of EES and TES:
    \begin{eqnarray}
      \label{eqn_EES}
      P_{\mathrm{ch},t}^{\rm EES}P_{\mathrm{dch},t}^{\rm EES}=0, \forall t \\
      \label{eqn_TES}
      H_{\mathrm{ch},t}^{\rm TES}H_{\mathrm{dch},t}^{\rm TES}=0, \forall t.
    \end{eqnarray}

  \item Constraints associated with shiftable loads \cite{8052160}:
  \begin{eqnarray}
    \label{eqn_shiftable_load_e}
    &\sum_{t\in \Omega_{\rm e}} {L_{\mathrm{e},t}^{\rm sl}}=L_{\rm e}^{\rm sl}\\
    &0 \leq L_{\mathrm{e},t}^{\rm sl} \leq \overline L_{\mathrm{e}}^{\rm sl}, \forall t\in  \Omega_{\rm e}\\\
    &\sum_{t\in \Omega_{\rm th}}{L_{\mathrm{th},t}^{\rm sl}}=L_{\rm th}^{\rm sl}\\
    & 0 \leq L_{\mathrm{th},t}^{\rm sl} \leq \overline L_{\mathrm{th}}^{\rm sl}, \forall t \in \Omega_{\rm th},
    \label{eqn_shiftable_load_h}
  \end{eqnarray}
  where $L_{\rm e}^{\rm sl}$ and $L_{\rm th}^{\rm sl}$ denote the total shiftable electric and heat load within the scheduling day, respectively; $\overline L_{\mathrm{e}}^{\rm sl}$ and $\overline L_{\mathrm{th}}^{\rm sl}$ are upper limits of the shiftable load; $\Omega_{\rm th}$ and  $\Omega_{\rm e}$ are feasible time intervals of shiftable electric load and thermal load, respectively.

  \item Upper limits of RES curtailment (which happens when the feed-in power hits the limit of the connecting line caused by too much local RES) and heat curtailment (which happens when the local heat output exceeds the local demand):
  \begin{eqnarray}
    \label{eqn_curtailment}
    0 \leq P_t^{\rm curt} \leq P_t^{\rm res},\forall t\\
    \label{eqn_curtailment_h}
    0 \leq H_t^{\rm curt},\forall t.
  \end{eqnarray}

  \item The upper and lower energy bounds of EES and TES, for $\forall t$:
    \begin{eqnarray}
      \label{eqn_EES_SOC}
    \underline E^{\rm EES} \leq (1-\alpha^{\rm EES})E_{\tc}^{\rm EES}
    + \sum_{\tau =\tc}^t \Delta E_{\tau}^{\rm EES} \leq  \overline E^{\rm EES}\\
    \label{eqn_TES_SOC}
    \underline E^{\rm TES} \leq (1-\alpha^{\rm TES} E_{\tc}^{\rm TES}
    + \sum_{\tau =\tc}^t \Delta  E{\tau}^{\rm TES} \leq  \overline E^{\rm TES}
  \end{eqnarray}
  where $\underline E^{\rm EES}, \underline E^{\rm TES}, \overline E^{\rm EES}, \overline E^{\rm TES}$ denote the minimum and maximum energy of EES and TES, respectively;
  $E_{\tc}^{\rm EES}$, $E_{\tc}^{\rm TES}$ denote the energy of the EES and TES at current period $\tc$; $\alpha^{\rm EES}$, $\alpha^{\rm TES}$ are the self-discharge rate of EES and TES; $\Delta E_t^{\rm EES}$ and $\Delta E_t^{\rm TES}$ are the net energy change of the EES and TES during the period $t$, which can be calculated respectively as:
  \begin{eqnarray}
    \label{eqn_EES_delta_energy}
    &\Delta E_t^{\rm EES}=\T\left(P_{\mathrm{ch},t}^{\rm EES}\eta_{\rm ch}^{\rm EES}- \frac{P_{\mathrm{dch},t}^{\rm EES}}{\eta_{\rm dch}^{\rm EES}}\right)\\
    &\Delta E_t^{\rm TES}=\T\left(H_{\mathrm{ch},t}^{\rm TES}\eta_{\rm ch}^{\rm TES}- \frac{H_{\mathrm{dch},t}^{\rm TES}}{\eta_{\rm dch}^{\rm TES}}\right),
  \end{eqnarray}
  where $\T$ denotes the scheduling period; $\eta_{\rm ch}^{\rm EES}$, $\eta_{\rm dch}^{\rm EES}$ , $\eta_{\rm ch}^{\rm TES}$ and $\eta_{\rm dch}^{\rm TES}$ denote charging and discharging efficiencies of EES and TES, respectively.

  \item The target energy constraints of storages to avoid end-of-horizon effects:
  \begin{eqnarray}
    \label{eqn_target_SOC_EES}
    E_{\tc}^{\rm EES} + \sum_{\tau =\tc}^{\te} \Delta E_{\tau}^{\rm EES} = E_{\rm targ}^{\rm EES}\\
    \label{eqn_target_SOC_TES}
    E_{\tc}^{\rm TES} + \sum_{\tau =\tc}^{\te} \Delta E_{\tau}^{\rm TES} = E_{\rm targ}^{\rm TES},
  \end{eqnarray}
  \end{enumerate}
  where $\te$ denotes the end period of optimization horizon; $E_{\rm targ}^{\rm EES}$ and $E_{\rm targ}^{\rm TES}$ are the target energy of EES and TES at the end of the optimization horizon.

  The autonomous optimization problem at $\tc$ for each MES can thus be formulated as:
  \begin{align}
    \label{eqn_auto_optimization}
    \begin{split}
      \textbf{P1: } & \min \sum_{t=\tc}^{\te}F_t\\
      & \mathrm{s.t.} \;\; \eqref{eqn_MESmodel}, \eqref{eqn_electricity}-\eqref{eqn_target_SOC_TES}.
    \end{split}
  \end{align}

  \subsection{Equivalent Energy Change Transform}
As can be seen, the problem \textbf{P1} is non-convex due to nonlinear terms in constraints \eqref{eqn_EES} and \eqref{eqn_TES}. This section endeavors to propose a method to relax these nonlinear constraints, such that \textbf{P1} can be convexified as:
  \begin{align}
    \label{eqn_auto_optimization_relaxed}
    \begin{split}
     \textbf{P2: } &\min \sum_{t=\tc}^{\te}F_t\\
      & \mathrm{s.t.} \;\;
      \eqref{eqn_MESmodel},
      \eqref{eqn_electricity}-
      \eqref{eqn_TES_power_limits},
      \eqref{eqn_shiftable_load_e}-
      \eqref{eqn_target_SOC_TES}.
    \end{split}
  \end{align}

  The relationship between optimal solutions of problem \textbf{P1} and \textbf{P2} is discussed below. Without loss of generality, we only discuss the constraint \eqref{eqn_EES}, i.e., the mutual exclusiveness of charging/discharging mode of the EES, and the same method is also applicable to the constraint \eqref{eqn_TES} for the TES.

  Let the feasible regions of  \textbf{P1} and \textbf{P2} be $K_1$, $K_2$, optimal solution vectors be $ X_1^*$, $X_2^*$, and optimal values be $f(X_1^*)$, $f(X_2^*)$, respectively. Besides, variables with a superscript * denote corresponding optimal values in $X_2^*$. For instance, $P_{\mathrm{ch},t}^{\rm EES*}$ and $P_{\mathrm{dch},t}^{\rm EES*}$ denote the optimal charging/discharging power of the EES in $X_2^*$ at time $t$.

  The following theorem provides a sufficient condition for equivalency of these two optimal solutions:

  \newtheorem{theorem}{Theorem}
  \begin{theorem}
    \label{thm_normal_mode}
    If no RES need to be curtailed, i.e., $P_t^{curt*}=0$, then $P_{\mathrm{ch},t}^{\rm EES*}P_{\mathrm{dch},t}^{\rm EES*} = 0, \forall t$ holds.
  \end{theorem}

   \begin{proof}
    To prove it by contradiction, suppose that $\exists t\in \left[\tc,\te\right]$, such that $P_{\mathrm{ch},t}^{\rm EES*} > 0,\;P_{\mathrm{dch},t}^{\rm EES*} > 0$. Denote the corresponding energy change of the EES as $\Delta E_t^{\rm EES*}$. Define the equivalent-energy-change (EEC) transform for the EES as:
  \begin{equation}
    \label{eqn_transform}
    \begin{split}
    \left(
      \widetilde P_{\mathrm{ch},t}^{\rm EES*}, \widetilde P_{\mathrm{dch},t}^{\rm EES*}
    \right)
    = \rm EEC(P_{\mathrm{ch},t}^{\rm EES*}, P_{\mathrm{dch},t}^{\rm EES*}) \\
    \triangleq \begin{cases}
      \left(\frac{\Delta E_t^{\rm EES*}}{\eta_{\rm ch}^{\rm EES}\T},\;0 \right),\Delta E_t^{\rm EES*}\geq 0 \\
      \left( 0,\; \frac{-\Delta E_t^{\rm EES*}\eta_{\rm dch}^{\rm EES}}{\T} \right),\Delta E_t^{\rm EES*}<0.
    \end{cases}
    \end{split}
  \end{equation}

  It is easy to verify that the new pair of charging/discharging power $(\widetilde P_{\mathrm{ch},t}^{\rm EES*}, \widetilde P_{\mathrm{dch},t}^{\rm EES*})$ leads to a same energy change of the EES as $\Delta E_t^{\rm EES*}$ during period $t$.

  The difference of EES's net discharge power can thus be calculated as:
  \begin{equation}
    \label{eqn_EES_delta_power}
    \begin{split}
        \Delta &P_{\mathrm{dch},t}^{\rm EES} =
        \left(
          \widetilde P_{\mathrm{dch},t}^{\rm EES*} - \widetilde P_{\mathrm{ch},t}^{\rm EES*}
          \right) -
          \left(
            P_{\mathrm{dch},t}^{\rm EES*} - P_{\mathrm{ch},t}^{\rm EES*}
            \right)\\
          &=\begin{cases}
            \left(
\left. 1 \right/ \eta_{\rm dch}^{\rm EES}\eta_{\rm ch}^{\rm EES} -1
            \right)
            P_{\mathrm{dch},t}^{\rm EES*},\Delta E_t^{\rm EES*}\geq 0 \\
        \left(
          1 - \eta_{\rm dch}^{\rm EES}\eta_{\rm ch}^{\rm EES}
        \right) P_{\mathrm{ch},t}^{\rm EES*},\Delta E_t^{\rm EES*} < 0
          \end{cases}\\
         &> 0,
        \end{split}
  \end{equation}
  which means that compared with $(P_{\mathrm{dch},t}^{\rm EES*},P_{\mathrm{ch},t}^{\rm EES*})$, an EES controlled according to $(\widetilde P_{\mathrm{dch},t}^{\rm EES*},\widetilde P_{\mathrm{ch},t}^{\rm EES*})$ always consumes less power (or discharge more power). Since there is no RES curtailment, according to \eqref{eqn_MESmodel} the system coordinator could purchase less electricity from the main grid by implementing the new charging/discharging power without breaking power balance. This contradicts with the fact that $X_2^*$ is an optimal solution of \textbf{P2}.
  \end{proof}

   According to Theorem \ref{thm_normal_mode}, without energy curtailment, the model has no incentive to charge and discharge simultaneously \cite{BECK2016331} and therefore $X_2^*=X_1^*$. Actually, references \cite{BECK2016331, 8486723} have conducted similar relaxation steps, but the precondition as well as the proof are not provided.

 However, in case local renewable energies are abundant, surplus power should be curtailed and this precondition is no longer established. In such case, apply the EEC transform and use the following substitutions in $X_2^*$:
  \begin{equation}
    \label{eqn_transformation}
    \begin{aligned}
      (P_{\mathrm{ch},t}^{\rm EES*},P_{\mathrm{dch},t}^{\rm EES*})
      \stackrel{\text{\eqref{eqn_transform}}}{\Longrightarrow}&
      (\widetilde P_{\mathrm{ch},t}^{\rm EES*},\widetilde P_{\mathrm{dch},t}^{\rm EES*}),\\
      P_t^{\rm curt*}
      \stackrel{\text{\eqref{eqn_EES_delta_power}}}{\Longrightarrow}&
      \widetilde P_t^{\rm curt*}=P_t^{\rm curt*}+ \Delta P_{\mathrm{dch},t}^{\rm EES},
    \end{aligned}
  \end{equation}
  and denote the new solution vector as $\widetilde X_2^*$.

  \begin{remark}
  \label{rem:transform}
   The EEC transform has no impact on the energy trajectory of an EES. The difference in the net discharge power caused by the transform is balanced by $\widetilde P_t^{\rm curt*}$, thus the exchanged electricity power $P_t^*$ of an MES is also kept unchanged.
  \end{remark}

  The sufficient condition for optimality of $\widetilde X_2^*$ is given in the following theorem:
  \begin{theorem}
    \label{thm_extreme_mode}
    If the optimal solution $X_2^*$ of \textbf{P2}  satisfies:
    \begin{equation}
      \label{eqn_condition_constraint}
      \frac{P_t^{\rm res} - P_t^{\rm curt*}}{1-\eta \rm _{ch}^{EES}\eta \rm _{dch}^{EES}} \geq \min (P_{\mathrm{ch},t}^{\rm EES*},\frac{P_{\mathrm{dch},t}^{\rm EES*}}{\eta \rm _{ch}^{EES}\eta \rm _{dch}^{EES}}),
    \end{equation}
    then $\widetilde X_2^*$ is an optimal solution of \textbf{P1}.
  \end{theorem}

  \begin{proof}
  See Appendix A.
  \end{proof}

  \begin{remark}
  The extreme circumstance when \eqref{eqn_condition_constraint} is unsatisfied is assumed to never occur due to the optimal planning procedure of energy systems in practice. As a consequence, the problem \textbf{P1} can be solved by solving a relaxed program \textbf{P2}. Although the EEC transform may be applied to the solution of \textbf{P2} as illustrated in Fig. \ref{fig_EEC_Flowchart}, it will not be explicitly stated in the subsequent models according to Remark \ref{rem:transform}.
  \end{remark}

  \begin{figure}
    \centering
    \includegraphics[width=0.4\textwidth]{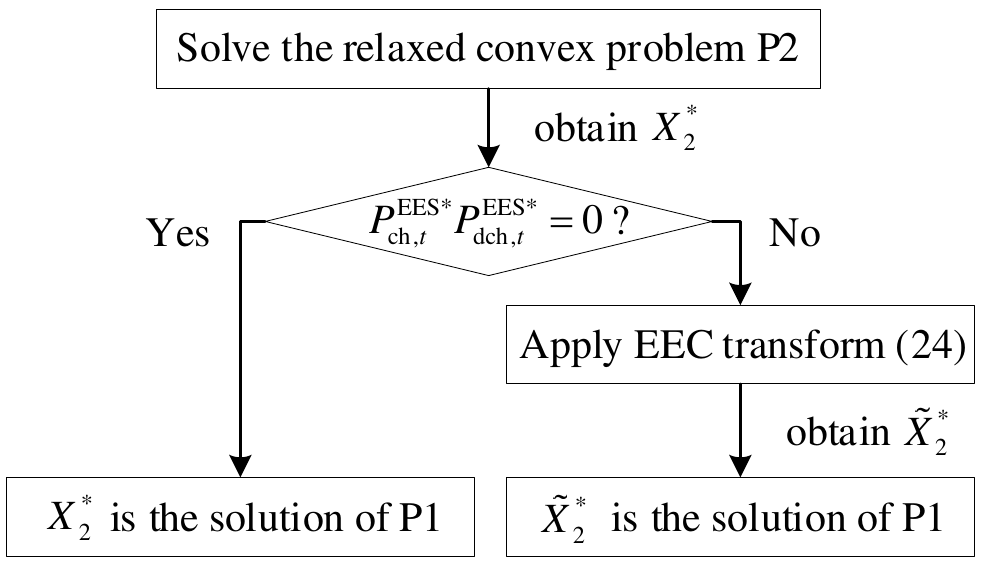}
    \caption{Applying the EEC transform}\label{fig_EEC_Flowchart}
  \end{figure}

  \section{Transactive Collaborative Optimization}


  \subsection{Bi-level Decomposition of Optimization Problem}
  During each period, the system coordinator aims at minimizing overall costs of IMESs across the remaining periods while keeping supply demand balanced and operation constraints satisfied. The optimization problem at $\tc$ can be modeled as follows:
  \begin{align}
    \label{eqn_col_optimization}
    \begin{split}
      \textbf{P3: }& \min \; \sum_{n=1}^{N}\sum_{t=\tc}^{\te}F_{t,n}\\
      & \mathrm{s.t.}  \;\; {\sum_{n=1}^{N}P_{t,n}}=P_t^{\rm Tr} + P_t^{\rm RES},\forall t\\
      & -\overline P^{\rm Tr,out}\leq P_t^{\rm Tr} \leq \overline P^{\rm Tr, in},\forall t\\
      & \text{
      \eqref{eqn_MESmodel}$_n$,
      \eqref{eqn_electricity}$_n$-
      \eqref{eqn_TES_power_limits}$_n$,
      \eqref{eqn_shiftable_load_e}$_n$-
      \eqref{eqn_target_SOC_TES}$_n$
      }, \forall n,
    \end{split}
\end{align}
  where $n$ is the index for MES; $N$ is total number of IMESs; $P_t^{\rm RES}$ is the total shared RES power (wind or solar) as shown in Fig \ref{fig_framework}; $P_t^{\rm Tr}$ is the imported power of the main transformer, and $\overline P^{\rm Tr, in},\overline P^{\rm Tr,out} \geq 0$ are the maximum exchange powers, determined either by the physical limit of the transformer or dispatch signal such as a peak-shaving request.

  The problem \textbf{P3} should have been solved in a centralized manner after gathering detailed information from all MESs. However, to preserve information privacy, this paper advocates to solve it in a distributed TE framework.

  The Lagrangian relaxed dual problem of \textbf{P3} is \cite{4118456}:
  \begin{align}
    \label{eqn_dual_problem}
    \begin{split}
      \textbf{P4: }&\;\;\max_{\forall t,\lambda _t} \varphi (\lambda_t)= \max \mathrm{inf}\, L\\
    &\mathrm{s.t.} -\overline P^{\rm Tr,out}\leq P_t^{\rm Tr} \leq \overline P^{\rm Tr, in},\forall t\\
    &\text{
      \eqref{eqn_MESmodel}$_n$,
      \eqref{eqn_electricity}$_n$-
      \eqref{eqn_TES_power_limits}$_n$,
      \eqref{eqn_shiftable_load_e}$_n$-
      \eqref{eqn_target_SOC_TES}$_n$
      }, \forall n,
    \end{split}
  \end{align}
  where $L$ is the Lagrangian function after introducing the Lagrange multipliers $\lambda_{\tc},\lambda_{t_{\rm c}+1},...,\lambda_{\te}$ associated with the power balance constraint:
  \begin{equation}
    \label{eqn_Lagrange_function}
    L={\sum_{n=1}^{N}\sum_{t=\tc}^{\te}F_{t,n}}
    +\sum_{t=\tc}^{\te}\lambda _t{(\sum_{n=1}^{N}P_{t,n}}-P{_t^{\rm Tr}}-P_t^{\rm RES}).
  \end{equation}

  Since the primal problem \textbf{P3} is linear, strong duality holds and the optimal value of \textbf{P4} is equivalent to that of \textbf{P3}.

  Under the given $\lambda_t$, the dual problem \textbf{P4} can be decomposed into $N+1$ subproblems (SP), corresponding to $N$ MESs:
  \begin{equation}
    \label{eqn_EH_SP}
    \begin{split}
      \forall n: \min &\sum_{t=\tc}^{\te}{\left[(\mu_{\mathrm{e},t}+\lambda_t)P_{t,n}+\mu_{\mathrm{g},t}(G_{\mathrm{g},t,n}^{\rm CHP}
    + G_{\mathrm{g},t,n}^{\rm GF}) \right]}\\
    &\mathrm{s.t.} \text{
      \eqref{eqn_MESmodel}$_n$,
      \eqref{eqn_electricity}$_n$-
      \eqref{eqn_TES_power_limits}$_n$,
      \eqref{eqn_shiftable_load_e}$_n$-
      \eqref{eqn_target_SOC_TES}$_n$
     },
  \end{split}
  \end{equation}
  and the transformer \cite{8274219}:
  \begin{equation}
    \label{eqn_TR_SP}
    \begin{split}
      &\min -\sum_{t=\tc}^{\te}\lambda_tP_t^{\rm Tr}\\
      \mathrm{s.t.} -&\overline P^{\rm Tr,out}\leq P_t^{\rm Tr} \leq \overline P^{\rm Tr, in},\forall t.
    \end{split}
  \end{equation}

  If the Lagrangian multiplier $\lambda_t$ in \eqref{eqn_EH_SP}, \eqref{eqn_TR_SP} is interpreted as price, a local price signal $\lambda _{\mathrm{e},t}$ can then be defined as:
  \begin{equation}
    \label{eqn_local_price}
    \lambda _{\mathrm{e},t} =\mu _{\mathrm{e},t} + \lambda _t.
  \end{equation}

  Thus, the multiplier $\lambda_t$ indicates the offset of the local electricity price $\lambda_{\mathrm{e},t}$ to the RTP price $\mu_{\mathrm{e},t}$, caused by the transformer congestion.

  Substitute \eqref{eqn_local_price} into \eqref{eqn_EH_SP} and we have:
  \begin{equation}
    \label{eqn_EH_SP_final}
    \begin{split}
      \forall n \textbf{ SP$_n$: } \min &\sum_{t=\tc}^{\te}{\left[\lambda_{\mathrm{e},t}P_{t,n}+\mu_{\mathrm{g},t}(G_{\mathrm{g},t,n}^{\rm CHP}
      + G_{\mathrm{g},t,n}^{\rm GF}) \right]}\\
      &\mathrm{s.t.} \text{
        \eqref{eqn_MESmodel}$_n$,
        \eqref{eqn_electricity}$_n$-
        \eqref{eqn_TES_power_limits}$_n$,
        \eqref{eqn_shiftable_load_e}$_n$-
        \eqref{eqn_target_SOC_TES}$_n$ .
        }
    \end{split}
  \end{equation}

  \textbf{SP$_n$} is exactly the individual optimization problem \textbf{P2} of MES$_n$ established in section III, except for that $\mu_{\mathrm{e},t}$ is replaced by $\lambda_{\mathrm{e},t}$.

  Similarly, substitute \eqref{eqn_local_price} into \eqref{eqn_TR_SP} and then decompose \eqref{eqn_TR_SP} into each control period, the subproblem \textbf{SP$_{N+1}$} is stated as:
  \begin{equation}
    \label{eqn_TR_SP_final}
    \begin{split}
    \textbf{SP$_{N+1}$: }\forall t, \min -(\lambda_{\mathrm{e},t}-&\mu_{\mathrm{e},t})P_t^{\rm Tr}\\
    \mathrm{s.t.} -\overline P^{\rm Tr,out}\leq &P_t^{\rm Tr} \leq \overline P^{\rm Tr, in}.
    \end{split}
  \end{equation}

  Obviously, the optimal solution of \textbf{SP$_{N+1}$} is:
  \begin{equation}
    \label{eqn_TR_SP_solution}
    \forall t, P_t^{\rm Tr}=
    \begin{cases}
      \overline P^{\rm Tr, in},&\lambda_{\mathrm{e},t}>\mu_{\mathrm{e},t}\\
      \forall,&\lambda_{\mathrm{e},t}=\mu_{\mathrm{e},t}\\
      -\overline P^{\rm Tr, out},&\lambda_{\mathrm{e},t}<\mu_{\mathrm{e},t},
    \end{cases}
  \end{equation}
  which indicates that when the local price is higher (lower) than the RTP price, the transformer purchases (sells) electricity from (to) the main grid as much as possible \cite{8274219} and vice versa.

  By far, problem \textbf{P3} can be solved in a bi-level framework. That is to say, at the upper level, the system coordinator adjusts the local price vector to strike a general balance between supply and demand. At the lower level, each MES autonomously minimizes its cost under the price vector. This process usually requires a large number of iterations. Denote the price vector in the $k$-th iteration by:
  \begin{equation}
    \label{eqn_forecast_price}
      \vec\Lambda_{\tc}^k = \left\{
      \lambda _{\rm e,\tc}^k,
      \lambda _{\rm e,\tc+1}^k,
      \dots,
      \lambda _{\mathrm{e},\te}^k
      \right\}\,.
  \end{equation}

  In each iteration, each MES solves \textbf{SP$_n$} with the given price vector $\vec\Lambda_{\tc}^k$ and bids the optimal power vector $\vec{P}_n^{k*}$, and the transformer bids the optimal power vector $\vec{P}^{\mathrm{Tr},k*}$ by solving \textbf{SP$_{N+1}$}. After receiving all biddings, the system coordinator computes the balance vector $\Delta \vec{P}^k = -\vec{P}^{\mathrm{Tr},k*} + \sum_{n=1}^{N}{\vec{P}_n^{k*}}$, and updates the price vector $\vec\Lambda_{\tc}^{k+1} = \vec\Lambda_{\tc}^k+\eta^k\Delta \vec{P}^k$, where $\eta^k$ denotes a feasible step length. These steps are repeated until the balance is achieved. In this paper, the above method is referred to as subgradient-based rolling TC (SG-RTC), since it adopts the dual subgradient method in a rolling optimization horizon.

  \begin{figure*}
    \centering
    \includegraphics[width=0.9\textwidth]{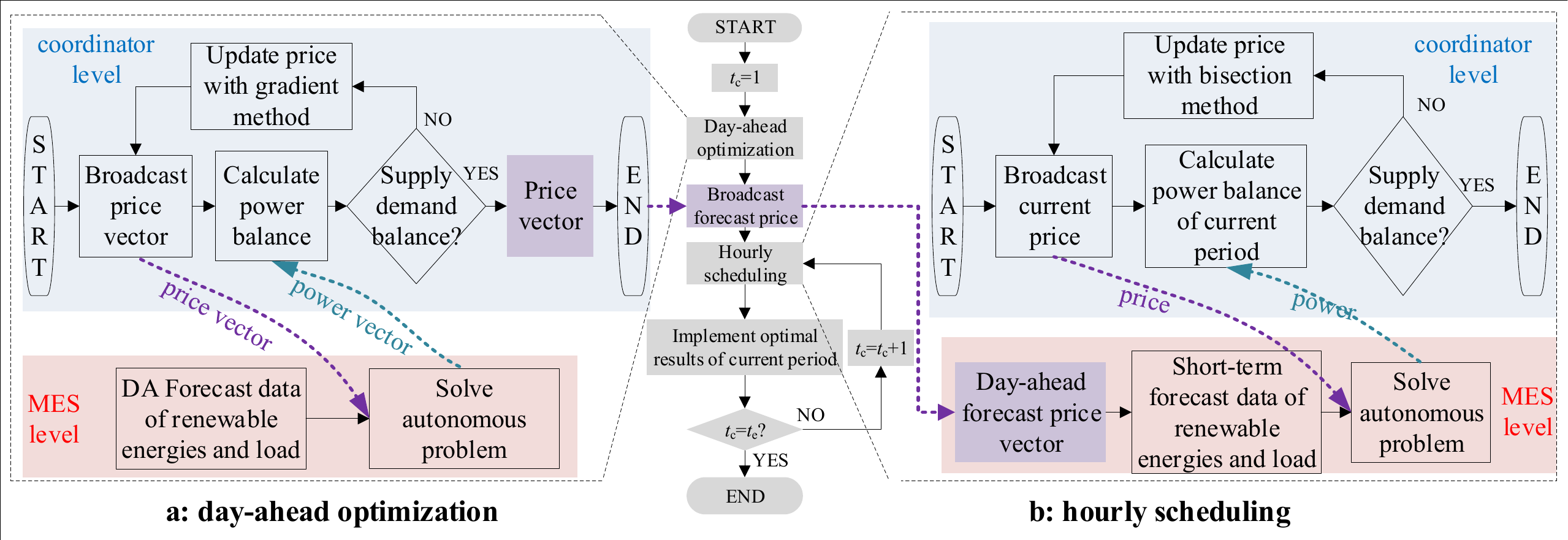}
    \caption{Operation framework of 2S-TC: day-ahead optimization (a) and intra-day hourly scheduling (b).}
    \label{fig_framework_two_stage}
  \end{figure*}

  \begin{table*}[]
    \caption{Comparisons between SG-RTC and 2S-TC}
    \label{table_difference}
    \centering
    \begin{tabular}{cc|lllll}
      \toprule
    \multicolumn{2}{c|}{TC} & main purpose & decision variable(s) & update method for multipliers & execution frequency \\
      \hline
    \multicolumn{2}{c|}{SG-RTC} & hourly dispatching & price vector & subgradient method & rolling (every period) \\
    \hline
    \multirow{2}{*}{2S-TC} & day-ahead stage & predicting next-day clearing prices & price vector  & subgradient method & once (non-rolling) \\
    & hourly stage & hourly dispatching & only current price & bisection method  & rolling (every period) \\
    \bottomrule
    \end{tabular}
  \end{table*}

  \subsection{Two-stage Optimization}

  Although the bi-level decomposition method significantly improves scalability, a large number of iterations required might be impractical for hourly-scheduling. To address this problem, a two-stage TC (abbreviated as 2S-TC) is proposed that splits the SG-RTC into two stages as illustrated in Fig.\ref{fig_framework_two_stage}. An day-ahead optimization is first implemented to forecast local electricity prices of the next day. Then the intra-day rolling optimization is implemented in hourly-scheduling to further deal with RES and loads' uncertainties. These two stages are compared in Table \ref{table_difference}, and the SG-RTC method proposed in \cite{8486723, 8501585, 8017431} is also compared in the table.

  \subsubsection{Day-ahead Stage}
  This stage is illustrated in Fig.\ref{fig_framework_two_stage}(a). Like the SG-RTC, the day-ahead stage also adopts the subgradient method to solve the intertemporal problem iteratively. However, since the purpose is to forecast local electricity prices of the next day, the rolling optimization strategy is not employed in this stage. Denote the forecast price vector of the next day in the $k$-th iteration by:
  \begin{equation}
    \label{eqn_forecast_price}
    \hat{\vec\Lambda} ^k = \left\{
      \hat\lambda _{\rm e,1}^k,
      \hat\lambda _{\rm e,2}^k,
      \dots,
      \hat\lambda _{\mathrm{e},\tc}^k,
      \dots,
      \hat\lambda _{\mathrm{e},\te}^k
      \right\}\,.
  \end{equation}

  Then a similar iterative process is implemented until the balance is achieved, and the final forecast price vector is denoted as $\hat{\vec\Lambda}$, which will be used in the hourly stage.

  \subsubsection{Hourly Stage}
  This stage aims to perform a hourly power adjustment based on ultra-short-term forecasts of RES and load, as illustrated in Fig.\ref{fig_framework_two_stage}(b).

  The idea of rolling optimization is employed in this stage. At period $\tc$, it is assumed that day-ahead forecasts of local electricity prices for future periods, i.e., $t_{\rm c}+1,t_{\rm c}+2,\dots,\te$, are perfect, and only the electricity price for the current period needs to be updated. Since now there is only one decision variable, a one-dimension search method such as bisection algorithm \cite{NumericalAnalysis} can be used. The bisection search space is denoted as $(\underline\lambda_{\rm e}, \overline\lambda_{\rm e})$ where $\underline\lambda_{\rm e}, \overline\lambda_{\rm e}$ are minimum and maximum prices respectively, and the search process is described as follows:

  S0: The system coordinator broadcasts the forecast price vector obtained in the day-ahead stage to all MESs:
  \begin{equation}
    \hat{\vec\Lambda} = \left\{
      \hat\lambda _{\rm e,1},
      \hat\lambda _{\rm e,2},
      \dots,
      \hat\lambda _{\mathrm{e},\tc},
      \dots,
      \hat\lambda _{\mathrm{e},\te}
      \right\}.
  \end{equation}

  S1: At the $p$-th iteration of period $\tc$, the coordinator broadcasts the price $\lambda_{\mathrm{e},\tc}^p$.

  S2: Each MES then generates the $p$-th price vector locally:
  \begin{equation}
    \label{eqn_RT_price}
    \vec\Lambda_{\tc}^p = \left\{
      \lambda_{\mathrm{e},\tc}^p,
      \overbrace{\hat\lambda _{\mathrm{e},t_{\rm c}+1},
      \dots,
      \hat\lambda _{\mathrm{e},\te}}^\text{day-ahead forecast prices}
      \right\},
  \end{equation}
 then solves its own subproblem \textbf{SP$_n$} and bids the optimal power, denoted as $P_{\tc,n}^{p*}$, to the coordinator.

 Meanwhile, the transformer bids its optimal power, denoted as $P_{\tc}^{\mathrm{Tr},p*}$, according to \eqref{eqn_TR_SP_solution}.

  S3: The coordinator calculates the power balance of the current period after receiving all bidding data:
  \begin{equation}
    \label{eqn_balance}
    \Delta P_{\tc}^p =
     -P_{\tc}^{\mathrm{Tr},p*}
     +\sum_{n=1}^{N}{P_{\tc,n}^{p*}}.
  \end{equation}

  S4: If $|\Delta P_{\tc}^p|$ is smaller than a predefined threshold $\zeta$, then the system converges and steps into S5. Otherwise, the coordinator updates  $\lambda_{\mathrm{e},\tc}^p$ according to the bisection method and steps back to S1.

  S5: Each MES implements local control according to its last bid. The above procedures are then repeated for the next control period.

  \section{Case Study}
  Three case studies will be conducted in this section. Case study I aims to compare the proposed 2S-TC with SG-RTC in terms of accuracy and scalability. Case study II will then focus on the 2S-TC framework. In the final case study, the EEC transform's effectiveness will be verified.

  Some common parameters are listed below. The control period is 1h. The maximum and minimum electricity price in the market are 1.0 and 0.2 yuan/kWh, respectively. The real-time prices are obtained from the PJM website \cite{pjmprice}. The price of natural gas is 3.3 yuan/m$^3$ \cite {8245724}. The day-ahead, intra-day and real-time forecast errors of RES are $\pm 30\%$, $\pm 10\%$ and $\pm 5\%$, respectively; the day-ahead, intra-day and real-time forecast errors of load are $\pm 20\%$, $\pm8\%$ and $\pm 3\%$, respectively \cite{20171225002}.

  \subsection{Case Study I: Comparisons of SG-RTC and 2S-TC}

  To verify performance of the proposed framework, this study conducts several large cases. Components in these cases will follow uniform distributions with parameters listed in Table \ref{table_component_parameters_case1}. The rolling horizon problem \textbf{P3} is solved with both SG-RTC and the proposed 2S-TC. As compared in Table \ref{table_difference}, prices of all remaining periods are iteratively updated with the subgradient method in SG-RTC, while the proposed 2S-TC would update the price of current interval in hourly-scheduling.

  \begin{table}[bt]
    \caption{Case Study I: Parameter ranges of components}
      \label{table_component_parameters_case1}
    \begin{tabular}{c|c|c|c|c|c}
    \toprule
    & parameter & value & & parameter & value
    \\ \hline \multirow{6}{*}{CHP} & capacity(MW) & 0-3 & \multirow{6}{*}{EES} & capacity(MWh) & 0-3
    \\ & \multirow{2}{*}{$\eta\rm _{gth}^{CHP}/\eta\rm _{gth}^{CHP}$} & \multirow{2}{*}{1-1.5} & & C-rate(C) & 0.1-0.3
    \\ & & & & $\rm E_{targ}^{EES}$(\%) & 15-50
    \\ & $\eta {\rm _{ge}^{CHP}}$(\%) & 25-40  & & $\underline E^{\rm EES}$, $\overline E^{\rm EES}$(\%) & 10,85
    \\ & $\underline P^{\rm CHP}$(\%) & 25-35 & & $\eta \rm_{ch}^{EES},\eta \rm_{dch}^{EES}$(\%) & 90
    \\ & $\Delta P^{\rm CHP}$(\%/h) & 30-50 & & $\alpha^{\rm EES}$(\%/d) & 0
    \\ \hline \multirow{2}{*}{GF} & installed heat & \multirow{2}{*}{0-1} &\multirow{5}{*}{TES} & capacity(MWh) & 0-3
    \\ & capacity(MW) & & & C-rate(C) & 0.1-0.3
    \\ & $\eta {\rm_{gth}^{GF}}$(\%) & 80-90 & & $\rm E_{targ}^{TES}$(\%) & 50-90
    \\ \cline{1-3} \multirow{3}{*}{EB} & capacity(MW) & 0.3-2 & & $\underline E^{\rm TES}$, $\overline E^{\rm TES}$(\%) & 10-90
    \\ & $\eta {\rm _{eth}^{EB}}$(\%) & 98 & & $\eta \rm_{ch}^{TES},\eta \rm_{dch}^{TES}$(\%) & 90
    \\ & $\Delta P^{\rm EB}$(\%/h) & 50 & & $\alpha^{\rm TES}$(\%/d) & 10
    \\ \bottomrule
    \end{tabular}
  \end{table}

  \subsubsection{Evaluation of Accuracy}
    A case consisting of 15 IMESs is considered here. Since SG-RTC's accuracy is guaranteed, it is viewed as a benchmark in this case. The overall cost is 247.18k yuan under 2S-TC and 247.17k yuan under the benchmark. Energy cost of each MES is listed in Table \ref{table_costs}. Results of the two methods are fairly close, suggesting that the proposed method is efficient in obtaining a rather optimal solution of the collaborative optimization.

  \begin{table}[]
    \centering
    \caption{Case Study I: Cost Comparisons of SG-RTC and 2S-TC (10$^3$ yuan)}
    \label{table_costs}
    \setlength{\tabcolsep}{1.2mm}{
    \begin{tabular}{@{}ccc|ccc|ccc@{}}
    \toprule
    MES & SG-RTC\tnote{*} & 2S-TC\tnote{**} & MES & SG-RTC & 2S-TC & MES & SG-RTC & 2S-TC \\
    \hline
    1   & 18.50 & 18.49 & 6  & 17.98 & 18.04 & 11 & 24.72 & 24.74 \\
    2   & 17.06 & 17.07 & 7  & 4.86 & 4.89 & 12 & 15.42 & 15.45 \\
    3   & 22.89 & 22.86 & 8  & 17.50 & 17.52 & 13 & 22.86 & 22.82 \\
    4   & 21.21 & 21.19 & 9  & 12.66 & 12.61 & 14 & 26.98 & 26.99 \\
    5   & 7.93  & 7.93 & 10 & 11.77 & 11.77 & 15 & 4.83 & 4.83 \\
    \bottomrule
    \end{tabular}}
  \end{table}

  \subsubsection{Evaluation of Scalability}
   To demonstrate scalability of the proposed method, cases with more IMESs are involved. The maximum and average number of iterations required by SG-RTC and 2S-TC during congestion periods are listed in Table \ref{table_iteration}. It can be concluded that the computational complexity will not increase significantly as the system scales up for both methods. However, 2S-TC method requires substantially fewer iterations to converge than SG-RTC in each period. Since the number of iterations corresponds to the communication costs between the system coordinator and MESs, and each MES is required to execute a local optimization for each iteration, 2S-TC can significantly improve control performances in terms of communication and computation requirements, especially considering the communication latency and distortion in practice.

\begin{table}[bt]
  \caption{Case Study I: Max (avg) number of iterations in hourly scheduling under SG-RTC and 2S-TC}
    \label{table_iteration}
    \centering
  \begin{tabular}{cccc}
  \toprule
  number of IMESs & 20 & 50 & 100 \\
  \midrule
  SG-RTC & 183 (88) & 183 (123) & 134 (98) \\
  2S-TC & 9 (6.5) & 9 (6.2) & 9 (6.5) \\
  \bottomrule
  \end{tabular}
  \end{table}

\subsection{Case study II: Evaluations of the 2S-TC}

  To facilitate detailed analysis of the collaborative autonomous optimization, a small system consisting of two residential MES (MES$_1$ and MES$_2$) and one commercial MES (MES$_3$) is simulated in this case. Parameters different from Table \ref{table_component_parameters_case1} are listed in Table \ref{table_component_parameters_case2}. The shared RES includes a 0.4 MW wind farm and a 0.3 MW solar farm. Profiles of fixed loads, shiftable loads and RES in a typical winter day are shown in Fig.\ref{fig_renewable_energy_load}. The shiftable electric load typically includes EVs, washing machines and dishwashers in residential MESs, as well as water heaters and disinfectors in the commercial MES. The power limit of the main transformer is 2.25MW. The connecting line limits of MES$_1$, MES$_2$ and MES$_3$ are 1.1 MW, 2.25 MW and 1.2 MW respectively.

  \begin{table}
    \centering
    \caption{Case Study II: Parameters of components}
      \label{table_component_parameters_case2}
    \begin{tabular}{|c|c|c|c|c|c|}
    \toprule
    \multirow{2}{*}{} & \multirow{2}{*}{parameter} & \multicolumn{3}{c|}{Case study II}
    \\ \cline{3-5} & & MES$_1$ & MES$_2$ & MES$_3$
    \\ \hline \multirow{5}{*}{CHP} & capacity(MW) & 1.5 &  \multirow{5}{*}{\textbackslash{}} & 4.0
    \\ \cline{2-3} \cline{5-5}& $\eta {\rm _{gth}^{CHP}}$(\%) & 42 & & 56
    \\ \cline{2-3} \cline{5-5}& $\eta {\rm _{ge}^{CHP}}$(\%) & 30 & & 28
    \\ \cline{2-3} \cline{5-5} & $\underline P^{\rm CHP}$(\%) & 30 & & 30
    \\ \cline{2-3} \cline{5-5} & $\Delta P^{\rm CHP}$(\%/h) & 40 & & 40
    \\ \hline \multirow{3}{*}{GF} & heat capacity(MW) & \multirow{3}{*}{\textbackslash{}} & 1.6 & \multirow{3}{*}{\textbackslash{}}
    \\ \cline{2-2} \cline{4-4} & $\eta {\rm_{gth}^{GF}}$(\%) & & 90 &
    \\ \cline{2-2} \cline{4-4} & $\underline P^{\rm GF}$(\%) & & 0 &
    \\ \hline \multirow{2}{*}{EB} & capacity(MW) & \multirow{2}{*}{\textbackslash{}} & 1 & \multirow{2}{*}{\textbackslash{}}
    \\ \cline{2-2} \cline{4-4} & $\underline P^{\rm EB}$(\%) & & 0 & 
    \\ \hline \multirow{3}{*}{EES} & capacity(MWh) & 1.6 & 1.5 &1.4
    \\ \cline{2-5} & C-rate(C) & 0.3 & 0.25 & 0.3
    \\ \cline{2-5} & $\rm E_{targ}^{EES}$(\%) &\multicolumn{3}{c|}{20}
    \\ \hline \multirow{3}{*}{TES} & capacity(MWh) & 1.2 & 1.2 & 1.4
    \\ \cline{2-5} & C-rate(C) & \multicolumn{3}{c|}{0.25}
    \\ \cline{2-5} & $\rm E_{targ}^{TES}$(\%) & \multicolumn{2}{c|}{60} & 50
    \\ \bottomrule
    \end{tabular}
  \end{table}

  \begin{figure}
    \centering
    \includegraphics[width=0.48\textwidth]{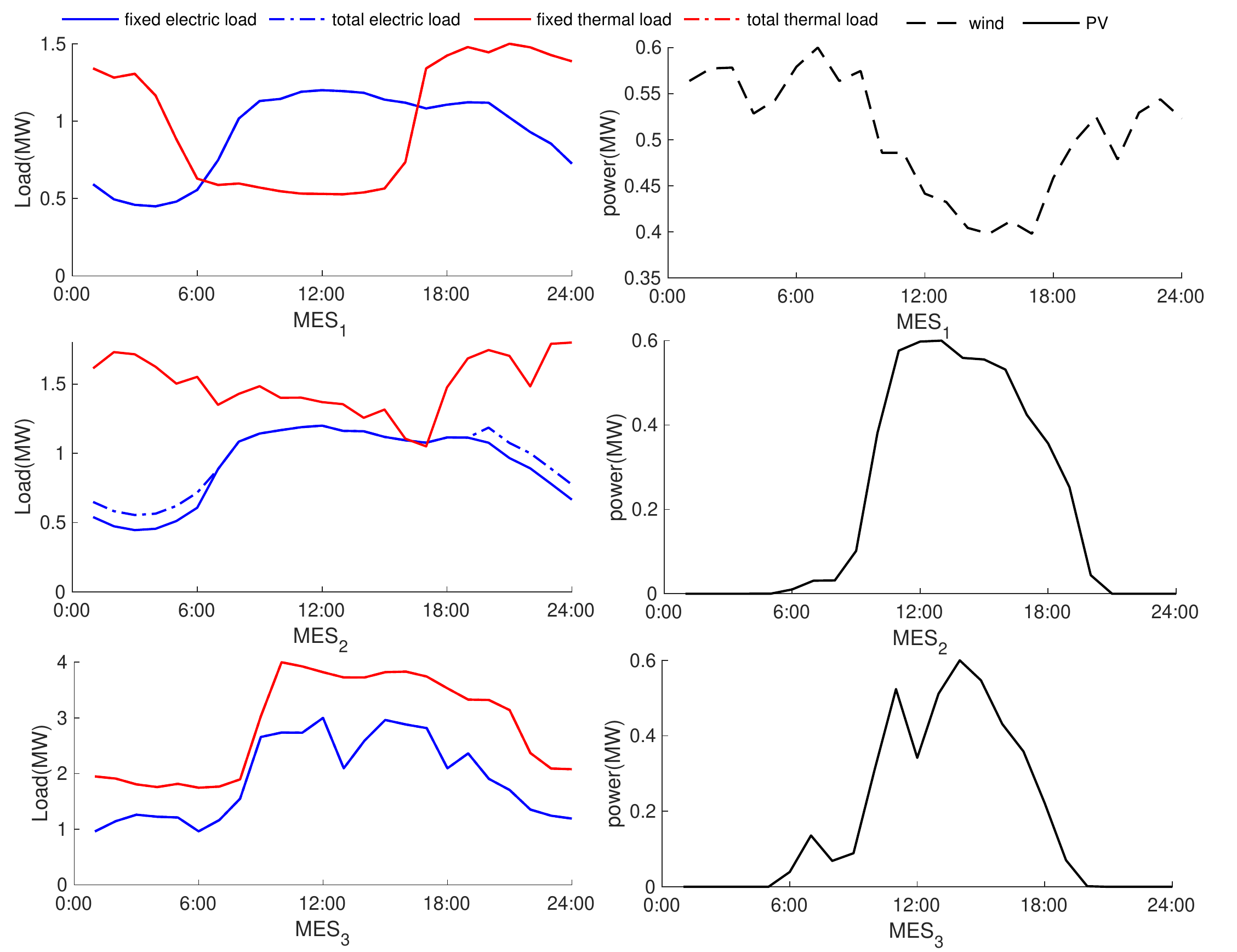}
    \caption{Case Study II: Curves of local loads and RES in each MES.}
    \label{fig_renewable_energy_load}
  \end{figure}

  For the convenience of analysis, the day is divided into two periods according to levels of price, i.e., valley-price hours (24:00-6:00) and peak-price hours (7:00-23:00).

\subsubsection{Effects of Collaborative Optimization}
  Three operation modes will be considered in this case study:
  \begin{enumerate}
    \item non-collaborative autonomous (NCA) mode: the IMESs are not coordinated by the system coordinator and each MES acts autonomously, i.e., each MES simply responds to RTP and self-optimizes according to \textbf{P1} \eqref{eqn_auto_optimization};
    \item collaborative autonomous (CA) mode: the autonomous MESs are coordinated under the 2S-TC framework. The whole IMESs system can both buy and sell electricity to maximize its profit. The feed-in electricity price is assumed to be same as RTP price;
    \item collaborative autonomous mode with feed-in limitation (CA-FIL): in contrast to CA mode, this mode aims to encourage local consumption of RES \cite{BECK2016331} and thus reduce the impact on the main power grid \cite{Feed-in-limitation}. This incentive purpose is usually achieved with the introduction of a feed-in price that is lower than the electricity price \cite{BECK2016331, 8501585}. In this mode the feed-in price is assumed to be zero.
  \end{enumerate}

  Simulation result of transformer power under these three modes are plotted in Fig.\ref{fig_transformer_3EH}. Total costs of the IMESs under three modes are 83.59k, 83.64k and 84.26k yuan respectively, while their RES accommodation rates are 87.69\%, 88.34\% and 100\%, respectively.

  In the NCA mode, it can be seen that the main transformer would suffer from overloading during 3:00-4:00 and 15:00-16:00, while the IMESs system would sell redundant RES back to the main grid during 10:00-11:00 and 18:00-20:00.

  After coordinating the IMESs in the CA mode, the main transformer is successfully protected from overloading. Meanwhile, power exportation through the main transformer can be observed in some periods to maximize the profit.

  In contrast, the CA-FIL mode can maximize local accommodation of RES while avoiding congestion through coordinating the IMESs. As a result, MESs are discouraged from arbitrage during 15:00-21:00 and a 100\% RES accommodation is reached in this condition. It should be noted that the IMESs' net demand are comparatively higher during valley hours and lower during peak hours in this mode. In other words, the overall electric demand pattern of the multi-energy systems contrasts with that of the main grid counterintuitively, as a response to the electricity price fluctuations.

  In the 2S-TC framework, the IMESs are coordinated using local price signals. To illustrate this, more details including the connecting line power of each MES are plotted in Fig.\ref{fig_MES_imported}.

  As illustrated in Fig.\ref{fig_MES_imported}(b), the CA mode will raise the clearing price above the RTP price when import congestion occurs during 2:00-4:00. As a result, MES$_2$ and MES$_3$ are discouraged from consuming electricity and would lower the EES charging power if possible during these hours, while MES$_1$ would increase the CHP output to make extra profit from selling electricity to other MESs. The congestion is finally relieved under the collaboration of all MESs successfully.

  In contrast,  while relieving the import congestion by raising the clearing price, CA-FIL mode will also cut down the clearing price when export congestion happens at 1:00, 9:00-12:00 and 23:00, as illustrated in Fig.\ref{fig_MES_imported}(c). Therefore, MES$_1$ and MES$_3$ are discouraged from generating much electricity through CHP during these hours and would lower the TES charging power correspondingly, while MES$_2$ is incentivized to consume more electricity to increase the overall RES accommodation within the IMESs.

  \begin{figure}
    \centering
    \includegraphics[width=0.48\textwidth]{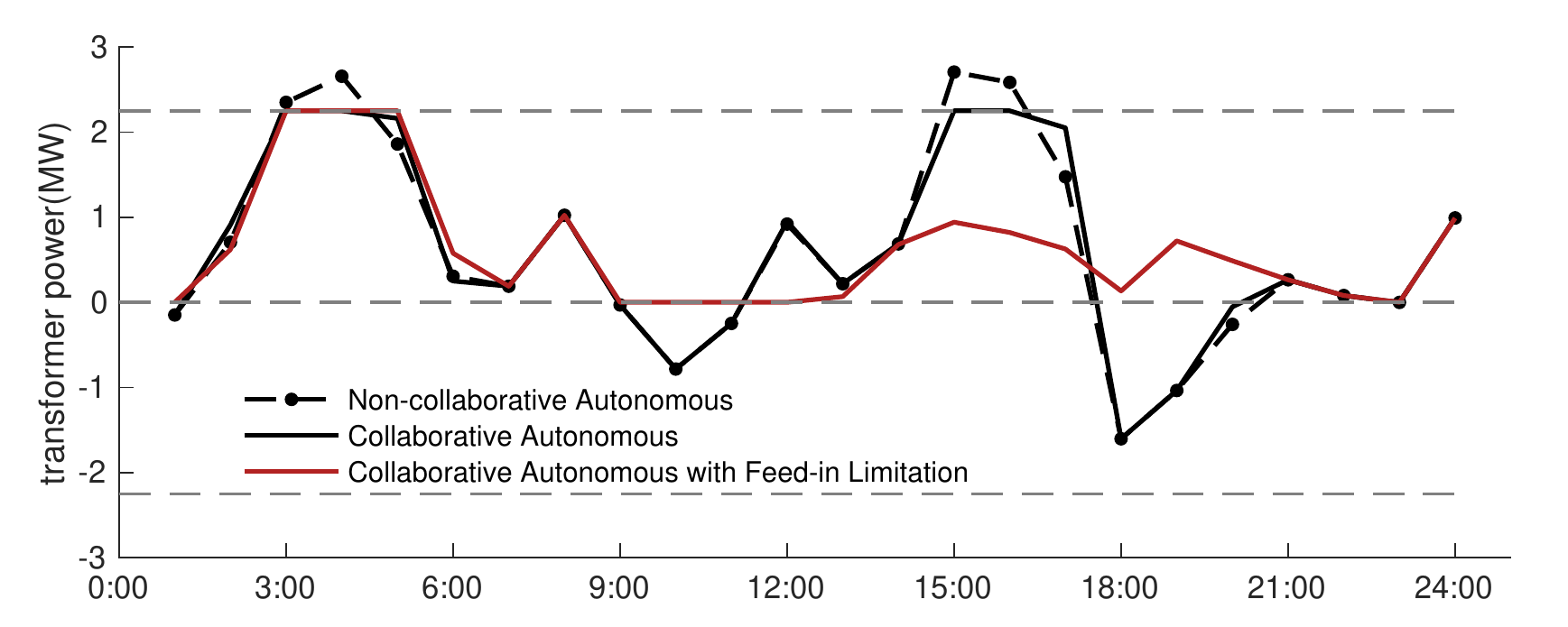}
    \caption{Case Study II: Transformer power under three modes.}
    \label{fig_transformer_3EH}
  \end{figure}

  \begin{figure}
    \subfigure[Non-collaborative Autonomous]{
    \includegraphics[width=0.47\textwidth]{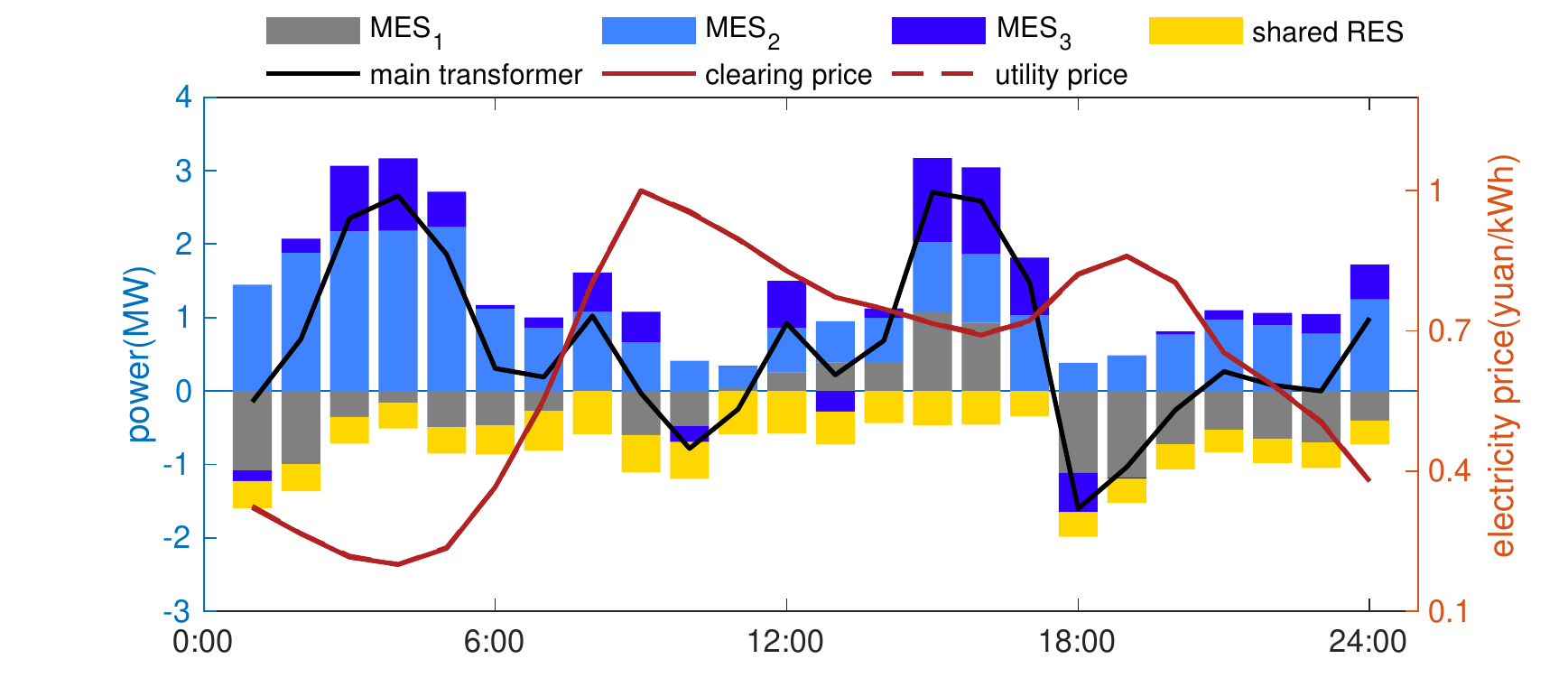}
    }
    \centering
    \subfigure[Collaborative Autonomous]{
    \includegraphics[width=0.47\textwidth]{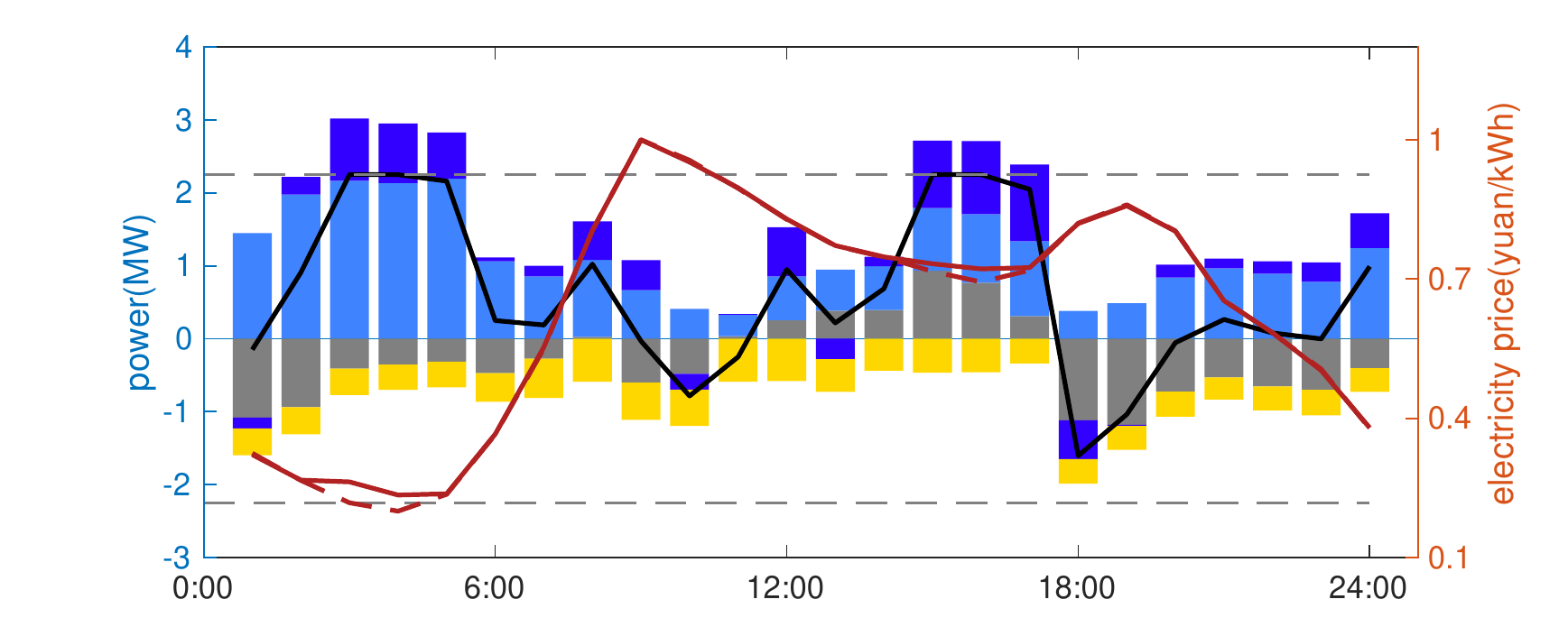}
    }
    \centering
    \subfigure[Collaborative Autonomous with Feed-in Limitation]{
    \includegraphics[width=0.47\textwidth]{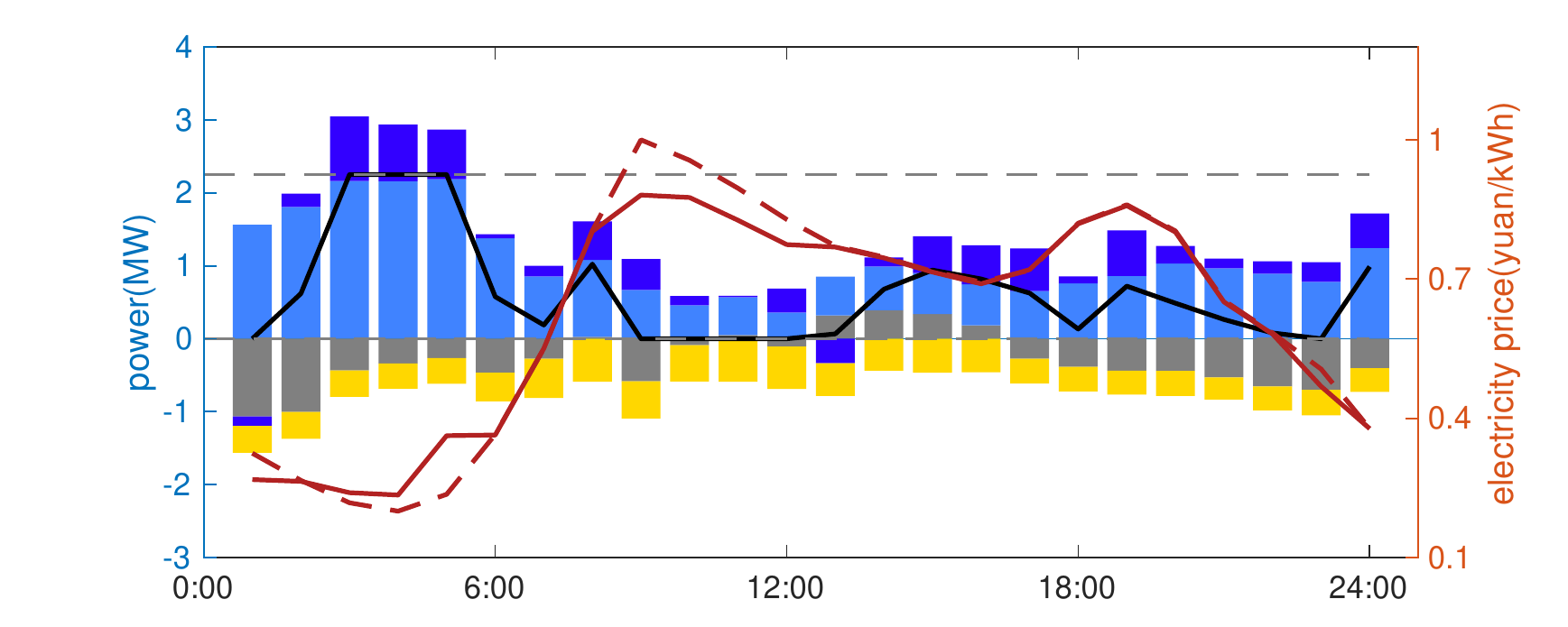}
    }
    \caption{Case Study II: Imported electricity of each MES and clearing price under three modes.}
    \label{fig_MES_imported}
  \end{figure}

\subsubsection{Effects of Autonomous Optimization}
  This section will focus on the MES-level self-optimization under CA-FIL mode. Results of all MESs are shown in Fig.\ref{fig_power_case2}. Note that for the bars in the figure, the generated power to supply demand is positive, while the exported or consumed power is negative. 

  \begin{figure}[htbp]
    \centering
    \subfigure[Electric power]{
      \includegraphics[width=0.45\textwidth]{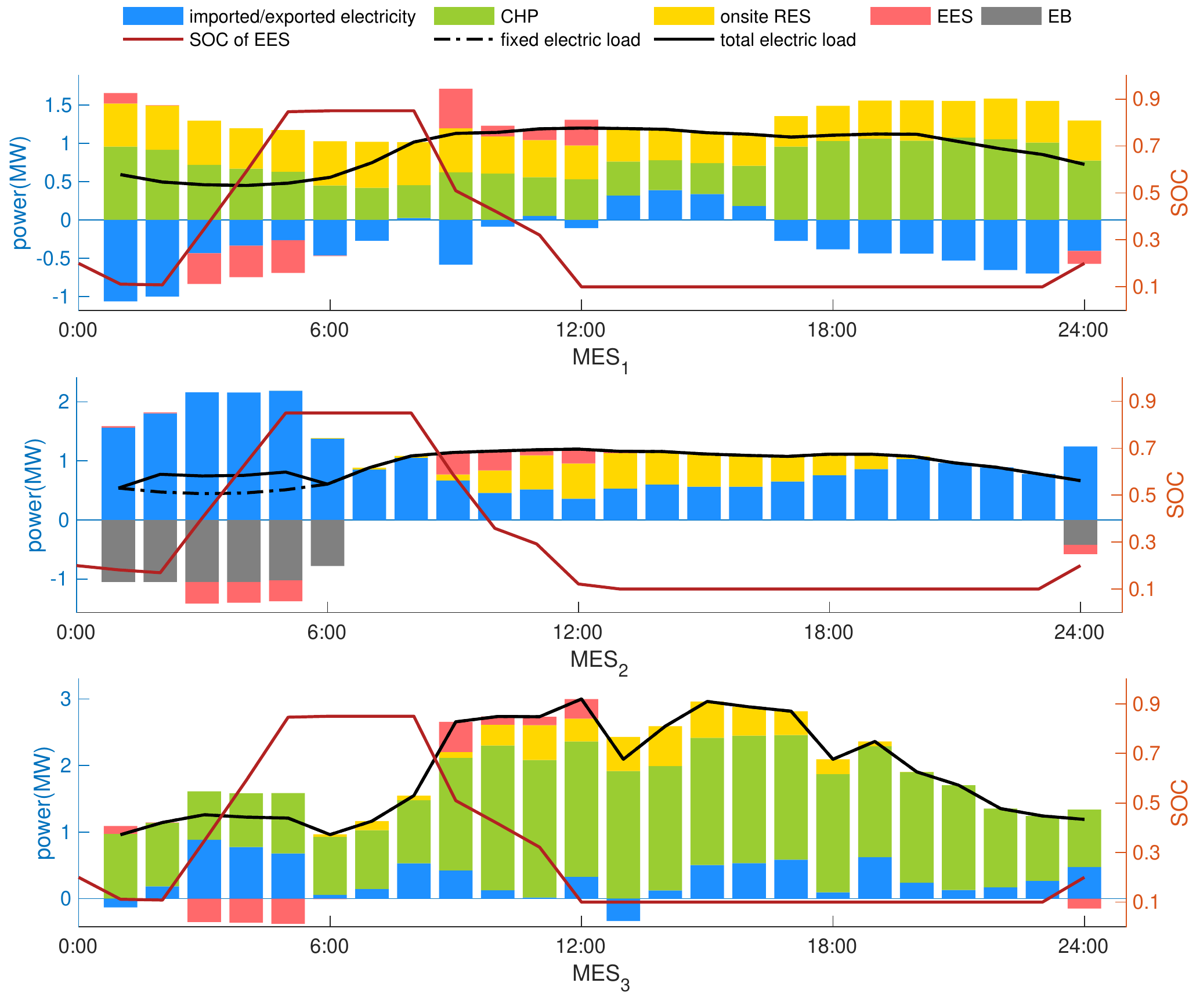}
    }
    \subfigure[Thermal power]{
      \includegraphics[width=0.45\textwidth]{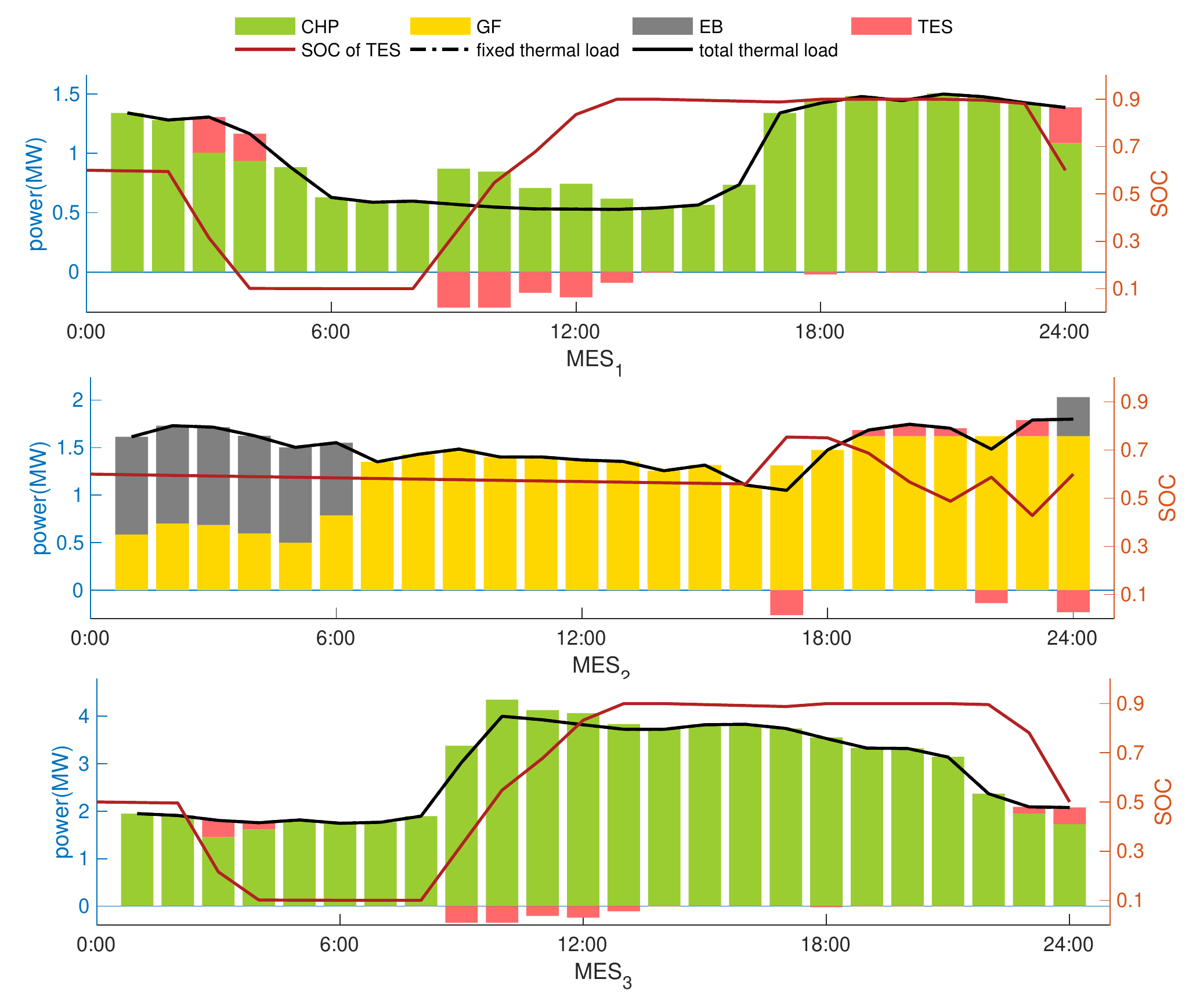}
    }
    \caption{Case Study II: Result of each MES' (a)electric and (b)thermal power.}
    \label{fig_power_case2}
  \end{figure}

  It can be seen from Fig.\ref{fig_power_case2} that the output level of CHP is determined by heat amount throughout the day in both MES$_1$ and MES$_3$. As a result, despite redundant wind resources at midnight in MES$_1$ and solar resources at noon in MES$_3$, excessive electricity are generated as by-product, so that these two MESs have to sell electricity to other MESs at that time. On the contrary, since no CHP is installed in MES$_2$, it has no alternative but to import electricity from the main grid to supply electric demand. Therefore, the electricity demand of MES$_2$ is comparatively higher all day long, especially during valley hours when heat are mainly supplied with the boiler. The following characteristics can thus be concluded from the simulation results:
  \begin{itemize}
    \item
    During valley hours, an MES tends to purchase cheap electricity from the main grid and supply heat with the boiler. During peak hours, it is more profitable to generate energy through CHPs and supply heat with the furnace.
    \item
    An MES tends to charge EES during valley hours and store electricity for peak hours, while a TES in MES installed with CHP (TES$_1$ and TES$_2$ in this case) tends to store thermal energy during peak hours and release heat during valley hours.
    \item
    The EES, TES and shiftable loads are coordinated such that the MES' net electric demand against thermal load matches the heat-to-electric ratio of CHP unit. On the one hand, it can minimize the use of furnace and reduce the cost of thermal energy during peak hours. On the other hand, under such conditions, curtailments of both electric and thermal power are minimized, thus the comprehensive efficiency of an MES is improved.
  \end{itemize}

To sum up, the complementary among multi-energy effectively help IMESs to better accommodate congestions caused by price fluctuation. It also gives full scope to each MES' optimization that aims to reduce its overall energy cost autonomously, while protecting the main transformer from overloading and increasing the RES accommodation, collaboratively.

\subsection{Case study III: Verification of the EEC Transform}

  To further verify the effectiveness of proposed EEC transform for EES, the installed capacity of wind turbines in one particular MES is intentionally enlarged in this section to arise wind curtailment at certain periods. Fig.\ref{fig_convex} plots optimal charging/discharging power of the EES obtained in \textbf{P1}, \textbf{P2} and (\textbf{P2}+EEC), which (\textbf{P2}+EEC) indicates the applying of EEC transformation on the optimal result of \textbf{P2} as illustrated in Fig. \ref{fig_EEC_Flowchart}. The optimal operational costs of MES obtained in both \textbf{P1} and (\textbf{P2}+EEC) are 15,512 yuan. The simulation results demonstrate that: (i) During time slots when no wind curtailment occurs, the optimal solution of \textbf{P2} always satisfies the mutual exclusiveness constraint, which verifies Theorem \ref{thm_normal_mode}. In contrast, at 6:00 and 24:00, abundant winds are curtailed and result of \textbf{P2} violates the relaxed constraints meanwhile. (ii) As expected, conditions where \eqref{eqn_condition_constraint} is unsatisfied have not been witnessed in the simulation case. Since the optimal operational costs of \textbf{P1} and (\textbf{P2}+EEC) are the same, the correctness of Theorem \ref{thm_extreme_mode} is also verified.

\begin{figure}
  \centering
  \includegraphics[width=0.48\textwidth]
  {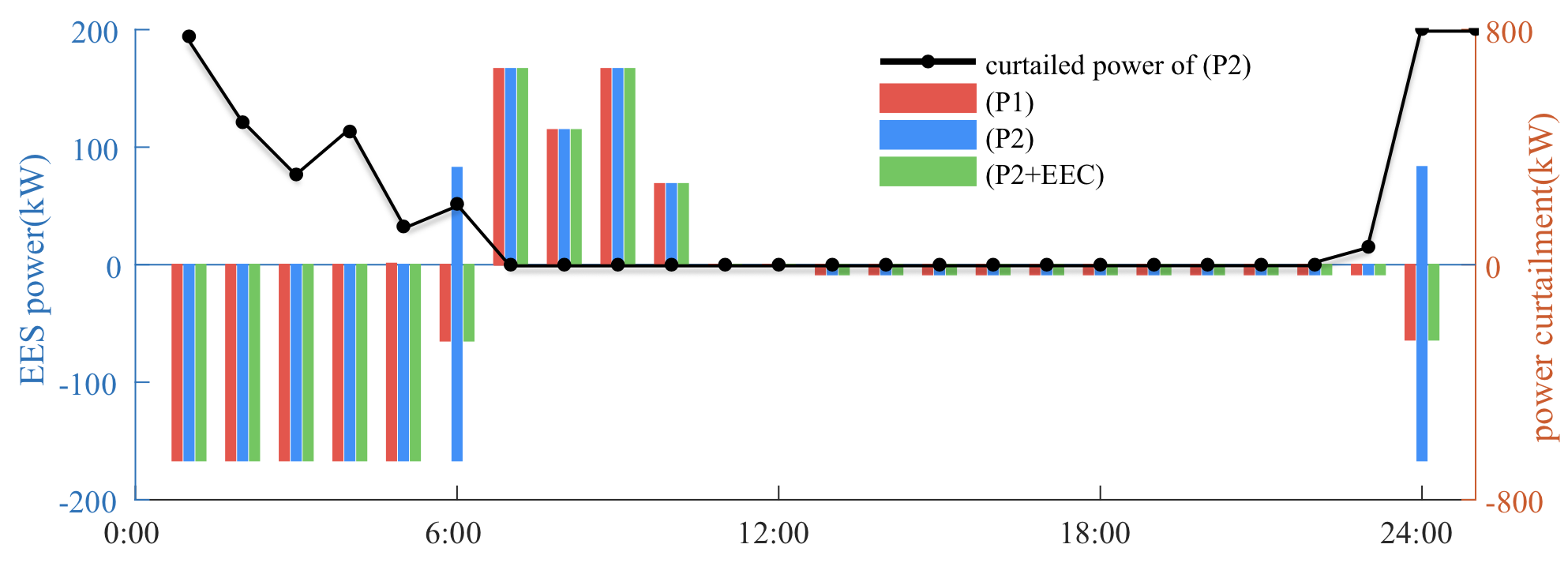}
  \caption{Power of RES curtailment and EES in a wind-rich MES}
  \label{fig_convex}
\end{figure}

\section{Conclusion}
Considering issues associated with information privacy and operation authority, this paper proposes a two-stage TC framework for coordinating IMESs that are managed by different management entities. The storages' complementarity constraints in autonomous optimization are firstly relaxed so that a global optimal solution is guaranteed, with its effectiveness verified in the simulation case. The two-stage transactive control framework is then established that solves the collaborative optimization in a distributed and scalable manner. Compared with sub-gradient based transactive control methods, it is verified to be efficient in obtaining a fairly optimal result of the rolling horizon optimization problem, while within substantially fewer iterations. Simulation shows that the dynamic coordination of large-scale IMESs enables each MES to reduce its overall energy cost autonomously, while protecting transformer from overloading and maximizing local accommodation of RES collaboratively.

Future work will include the coordination of distributed energy resources in the lower level based on transactive control and the analysis of establishing a multi-energy market.

\begin{appendices}

\section{Proof of theory 2.}

 \begin{proof}
  Condition 1: When $X_2^* \in K_1$. First, since \eqref{eqn_EES} is satisfied and the right-hand-side of \eqref{eqn_condition_constraint} equals zero, condition \eqref{eqn_condition_constraint} is always met. Second, because $X_2^*$ and $X_1^*$ denote the feasible and optimal solution of \textbf{P1} respectively, it can be derived that $f(X_1^*) \leq f(X_2^*)$. Meanwhile, since $K_1 \subset K_2$, then $f(X_1^*) \geq f(X_2^*)$. Therefore, $f(X_1^*)=f(X_2^*)$ and $X_2^*$ is also an optimal solution of \textbf{P1}. At last, it is obvious that $X_2^* = \widetilde X_2^*$. To sum up, $\widetilde X_2^*$ is an optimal solution of \textbf{P1}.

  Condition 2: When $X_2^* \notin K_1$, which means that $\exists t\in [\tc,\te]$, such that $P_{\mathrm{ch},t}^{\rm EES*}P_{\mathrm{dch},t}^{\rm EES*} > 0$.
  To prove that $\widetilde X_2^*$ is an optimal solution of \textbf{P1}, let's first prove that it is a feasible solution.

  Without loss of generality, assume $\Delta E_t^{\rm EES*} \geq 0$, and similar proof can be derived when $\Delta E_t^{\rm EES*} < 0$.

  First, $\widetilde X_2^*$ apparently satisfies constraints\eqref{eqn_electricity}-\eqref{eqn_GF}, \eqref{eqn_TES_power_limits}, \eqref{eqn_TES}-\eqref{eqn_shiftable_load_h}, \eqref{eqn_curtailment_h}, \eqref{eqn_TES_SOC}, \eqref{eqn_target_SOC_TES} since the modified variables are not included in these constraints. Besides, the transformation method introduced in \eqref{eqn_transformation} guarantees that $\widetilde X_2^*$ always meets the power balance constraint \eqref{eqn_MESmodel} and the mutual exclusiveness constraint \eqref{eqn_EES}. As for constraints \eqref{eqn_EES_power_limits}, \eqref{eqn_curtailment}, \eqref{eqn_EES_SOC}, \eqref{eqn_target_SOC_EES}:

  \begin{enumerate}
    \item The maximum and minimum charging/discharging power constraint of EES.
    According to \eqref{eqn_transform}:
    \begin{align}
      \begin{split}
      0 \leq& \Delta E_t^{\rm EES*}/\eta_{\rm ch}^{\rm EES}
      = \widetilde P_{\mathrm{ch},t}^{\rm EES*} \\
      =& P_{\mathrm{ch},t}^{\rm EES*}-
      \frac{P_{\mathrm{dch},t}^{\rm EES*}}{\eta \rm_{dch}^{EES}}
      < P_{\mathrm{ch},t}^{\rm EES*}
      \leq \overline P_{\rm ch}^{\rm EES}.
      \end{split}
    \end{align}
    Therefore, constraint \eqref{eqn_EES_power_limits} is satisfied.

    \item Upper and lower limits of RES curtailment. According to \eqref{eqn_EES_delta_power} and \eqref{eqn_condition_constraint}:
    \begin{align}
      \begin{split}
      0 \leq &
      P_t^{\rm curt*} <
      \widetilde P_t^{\rm curt*} \\=
      &P_t^{\rm curt*} +
      \left(
          \frac{1}{\eta_{\rm dch}^{\rm EES}\eta_{\rm ch}^{\rm EES}}-1
        \right)
        P_{\mathrm{dch},t}^{\rm EES*}
      \leq P_t^{\rm res}.
      \end{split}
    \end{align}
    Therefore, constraint\eqref{eqn_curtailment} is also satisfied.

    \item The upper and lower bounds of EES energy. Since $\Delta E_t^{\rm EES*} \equiv \Delta\widetilde E_t^{\rm EES*} $, the energy change of the EES during period $t$ stays unchanged and thus constraint \eqref{eqn_EES_SOC} and \eqref{eqn_target_SOC_EES} still holds.
  \end{enumerate}

  Therefore, all constraints of \textbf{P1} are met for $\widetilde X_2^*$, and $\widetilde X_2^*$ is a feasible solution of \textbf{P1}. Since $f(X_2^*) = f(\widetilde X_2^*)$, hereafter conclusion of condition 1 can be applied to derived that $f(X_1^*) = f(\widetilde X_2^*)$. Thus, $\widetilde X_2^*$ is an optimal solution of \textbf{P1}.
  \end{proof}

  \end{appendices}



  \bibliographystyle{IEEEtran}
  \bibliography{references}


  \end{document}